\documentclass[12pt]{article}
\usepackage{graphicx} 
\usepackage{caption}     
\usepackage{subcaption, comment} 
\usepackage{amsfonts,amsmath,amssymb,amsthm, mathrsfs, dsfont, enumitem}
\usepackage{xcolor}
\usepackage{hyperref}
\usepackage[noabbrev,capitalize,nameinlink]{cleveref}

\theoremstyle{plain}
\newtheorem{theorem}{Theorem}[section]
\newtheorem*{theorem*}{Theorem}
\newtheorem{lemma}[theorem]{Lemma}

\newtheorem{proposition}[theorem]{Proposition}
\newtheorem*{proposition*}{Proposition}

\theoremstyle{definition}

\theoremstyle{remark}

\crefname{lemma}{Lemma}{Lemmas}
\crefname{theorem}{Theorem}{Theorems}
\crefname{proposition}{Proposition}{Propositions}
\crefname{enumi}{Statement}{Statements}
\crefname{claim}{Claim}{Claims}

\newcommand*{\1}{\mathds{1}}

\usepackage[style=numeric, url=false, doi=false, isbn=false, maxbibnames=10]{biblatex}
\title{Stability of the Rankine Vortex and Perimeter Growth in Vortex Patches}
\author{John Brownfield}
\date{November 2025}

\begin{document}

\maketitle

\setlength{\parindent}{0pt}

\begin{abstract}
    We prove that for $\omega: \mathbb{R}^2 \to [0,1]$ sharing the same total vorticity and center of vorticity as the Rankine vortex, the $L^1$ deviation from the Rankine patch can be bounded by a function of the pseudo-energy deviation and the angular momentum of $\omega$. In the case of $m-$fold symmetry, the dependence on the angular momentum can be dropped. Using this, we affirm the results of prior simulations by demonstrating linear in time perimeter growth for a simply connected perturbation of the Rankine vortex.
\end{abstract}

\tableofcontents

\section{Introduction}
The vorticity form of 2D incompressible Euler is
\begin{align*}
    \partial_t \omega + u\cdot\nabla\omega &= 0 \\
    \omega |_{t=0} &= \omega_0,
\end{align*}
where $u$ is given by the Biot-Savart Law as $u = K*\omega$,
\begin{align*}
    K(x) := \frac{x^\perp}{2\pi|x|^2} = \frac{1}{2\pi|x|^2}(-x_2, x_1).
\end{align*}
It was shown by Yudovich \cite{Yudovich_1963} that a unique global in time weak solution $\omega$ exists under the assumption $\omega_0 \in L^1 \cap L^\infty$. An interesting class of solutions, called vortex patches, are those for which $\omega_0$ is the characteristic function of a set of finite measure. Since $\omega$ satisfies a transport equation, this initial condition is preserved for all time, and Chemin \cite{Chemin_1991} showed that smooth initial boundaries remain smooth. Despite this, an interesting topic of recent research has been constructing solutions that lose initial regularity in infinite time. These solutions are often found as stable perturbations (in weak norms) of known, simple solutions, and the stability is often accomplished by leveraging conserved quantities. For 2D incompressible Euler, the following quantities,
\begin{align*}
    \int \omega, \int x\omega, \int|x|^2\omega, \text{ and } \int |u|^2,
\end{align*} 
are all conserved when finite. These are respectively referred to as the total vorticity, center of vorticity, angular momentum, and kinetic energy. For fixed total vorticity and center of vorticity at the origin, it is not hard to see that a disk is the unique minimizer of the angular momentum among vortex patches. As a result, $\omega_* = \1_{B(0,r)}$ is a steady state of 2D incompressible Euler, and since the equations are translation invariant, so is $\omega_* = \1_{B(a,r)}$. Solutions of this form are known as Rankine vortex patches, and they enjoy a long history of stability results. Pulvirenti-Wan \cite{Wan_1985} used this characterization of Rankine patches as minimizers of the angular momentum to derive quantitative $L^1$ stability estimates on bounded domains. Sideris-Vega \cite{Sideris_2009} later upgraded this result to the whole plane, essentially bounding the $L^1$ deviation of $\omega_t$ from a disk by the angular momentum of $\omega_0$ that came from outside the Rankine patch. \\

A similar idea works with the kinetic energy, after a suitable modification. The problem is that for a nontrivial vortex patch, $u$ decays like $1/|x|$ and is not in $L^2(\mathbb{R}^2)$. However a related conserved quantity, the pseudo-energy, is both finite and maximized (under fixed total vorticity and center of vorticity) among vortex patches by a disk. The pseudo-energy $E(\omega)$ is found by integrating by parts the kinetic energy:
\begin{align*}
    E(\omega) := \frac{1}{2\pi}\int \int_{\mathbb{R}^2\times \mathbb{R}^2} \omega(x) \ln(|x-y|^{-1})\omega(y)dxdy.
\end{align*}
Pulvirenti-Wan \cite{Wan_1985} was one of the first to use the pseudo-energy to develop a quantitative $L^1$ stability result for Rankine vortex patches on disk shaped domains. Tang \cite{Tang_1987} upgraded this to the whole plane by showing that, assuming an initial patch sufficiently close to the Rankine vortex and with bounded angular momentum, the $L^1$ deviation can still be bounded by that of the pseudo-energy. In the time since, the methods used for deriving quantitative estimates for extrema of these convolution-like functionals have developed greatly. It was remarked in the work of Yan-Yao \cite{Yan_2022} that the results of Frank-Lieb \cite{Frank_2021} imply the following proposition:
\begin{proposition} \label{prop:pseudo-bound}
    There exists $C_n>0$ such that for all $\omega : \mathbb{R}^n \to [0,1]$ with $\omega \in L^1(\mathbb{R}^n)$,
    \begin{align*}
        \int\int_{E^*\times E^*}& \ln(|x-y|^{-1})dxdy - \int\int_{\mathbb{R}^n\times\mathbb{R}^n} \omega(x)\ln(|x-y|^{-1})\omega(y)dxdy \\
        &\geq C_n \inf_{a\in\mathbb{R}^n}||\omega-\1_{E^*+a}||_{L^1}^2
    \end{align*}
    where $E^*$ is the ball centered at the origin with $|E^*| = \int_{\mathbb{R}^n} \omega(x)dx$.
\end{proposition}

A proof of \cref{prop:pseudo-bound} using the results of \cite{Frank_2021} is included in the appendix. Using this proposition, we can show the following two stability theorems for Rankine vortex patches:
\begin{theorem}[Stability under $m$-fold symmetry]\label{thm:mstability}
    There exists a universal constant $C>0$, such that if $\omega_0: \mathbb{R}^2 \to [0,1]$ satisfies $\omega_0 = \omega_0 \circ R_{2\pi/m}$ for some $m \geq 2$ and $\int \omega_0(x)dx = \pi r^2$, then the solution to 2D incompressible Euler having initial data $\omega_0$ satisfies
    \begin{align*}
        ||\omega_t - \1_{B(0,r)}||^2_1 \leq C(E(\1_{B(0,r)}) -E(\omega_0))
    \end{align*}
    for all time.
\end{theorem}

\begin{theorem}[Stability under bounded angular momentum]\label{thm:istability}
    Let $I, r >0$. There exists $C = C(I,r) > 0$ such that if
    $\omega_0:\mathbb{R}^2 \to [0,1]$ satisfies 
    \[
    \int |x|^2 \omega_0(x)dx \leq I, \int x\omega_0(x)dx =0,\text{ and } \int \omega_0(x)dx = \pi r^2,
    \]
    then the solution to 2D incompressible Euler having initial data $\omega_0$ satisfies
    \begin{align*}
        ||\omega_t - \1_{B(0,r)}||^4_1 \leq C(E(\1_{B(0,r)}) -E(\omega_0))
    \end{align*}
    for all time.
\end{theorem}

It should be noted the way in which these theorems extend the results of \cite{Tang_1987}. In both cases we removed the restriction on the initial $L^1$ distance and relaxed the patch assumption to $0 \leq \omega \leq 1$. In the case of $m$-fold symmetry, we further removed the bound's dependence on angular momentum and improved the exponent. \\

Returning to regularity loss, our main result is the construction of a simply connected vortex patch whose perimeter grows linearly in time. Absent the simply connected condition, Drivas-Elgindi-Jeong \cite{Drivas_2024} provided the first example of such a patch via a perturbation of the Rankine vortex having multiple connected components. Their result rested on an $L^2$ estimate for the stability of the differential shearing of trajectories that occurs with the Rankine patch. However, the major difficulty in applying this $L^2$ estimate towards perimeter growth is that the boundary is a zero measure set. \cite{Drivas_2024} handles this problem by having the perimeter enclose regions of both the patch and its complement, forcing the shearing to stretch the perimeter. \\

Such a method cannot extend to producing simply connected examples, however the stability of the differential shearing can still be exploited. The main concern to overcome is that somehow the bulk of the mass inside and outside the patch will be able to slide past the boundary as it shears, leaving the perimeter unchanged. To address this, we use \cref{thm:mstability} to construct a stable $L^1$ perturbation of the Rankine patch with arbitrarily large angular momentum. This forces some part of the perimeter to always be far from the origin, creating a natural barrier for points in the complement of the patch that are experiencing shearing. A topological argument involving the lifted dynamics to the universal cover of $\mathbb{R}^2 \setminus\{0\}$ shows how this forces shearing to occur on the perimeter. Roughly, the exterior of the patch in the universal cover can be split into buckets (see \cref{fig:decomp}) uniquely associated to each lift of the boundary, and the slow movement of the interfaces between buckets limits the rate at which points can move past the boundary of the patch. Putting it all together, we get our main result:

\begin{theorem}[Linear in Time Perimeter Growth]\label{thm:pgrowth}
    There exists a simply connected open set $\Omega_0 \subset \mathbb{R}^2$ (with smooth boundary) such that the vortex patch solution $\omega_t = \1_{\Omega_t}$ to 2D incompressible Euler with initial data $\omega_0 =\1_{\Omega_0}$ satisfies
    \[
        \text{Perimeter}(\Omega_t) \gtrsim t
    \]
    for $t$ sufficiently large.
\end{theorem}

 A sketch of our patch is included in \cref{fig:patch}, where $N >> 1$. The simple shape of our patch could point to how the filaments that appeared in the simulations of Dritschel \cite{Dritschel_1988} contribute to continual perimeter growth. Interestingly, our method only relies on $m$-fold symmetry, large angular momentum, and pseudo-energy sufficiently close to that of the Rankine vortex to show linear in time perimeter growth for the boundary of the patch. 

\begin{figure}[htbp]
  \centering
  \includegraphics[width=0.3\linewidth]{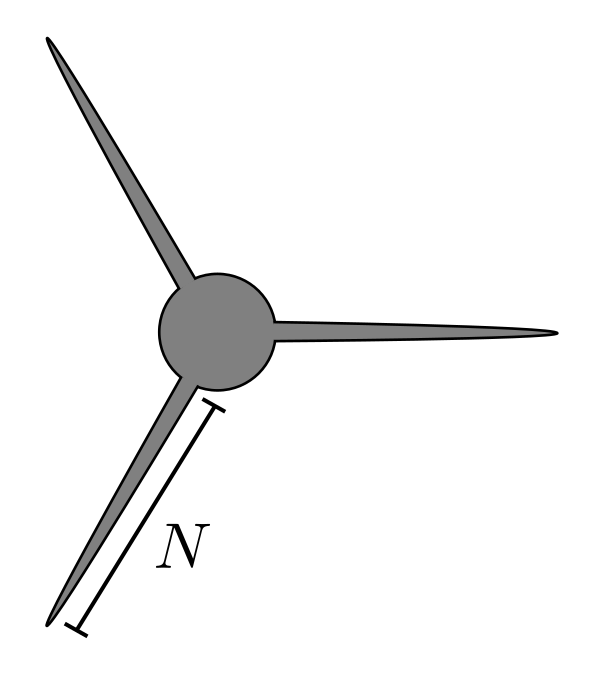}
  \caption{A rough sketch of $\Omega_0$}
  \label{fig:patch}
\end{figure}

\section{Acknowledgements}
The author thanks T. Elgindi for both introducing him to this problem as well as for his helpful discussion and comments. The author also thanks T. Drivas and I. Jeong for comments that have improved this paper. The author was partially supported by the NSF-DMS grant 2038056.

\section{Stability of the Rankine Vortex}
The core of the proofs of \cref{thm:mstability} and \cref{thm:istability} is contained in the following two propositions:
\begin{proposition}\label{prop:mbound}
    Suppose $\omega:\mathbb{R}^2 \to [0,1]$ satisfies $\omega = \omega \circ R_{2\pi/m}$ for some $m \geq 2$, and $\omega \in L^1(\mathbb{R}^2)$. Then
    \begin{align*}
        \inf_{a\in\mathbb{R}^2}||\omega-\1_{E^*+a}||_{L^1} \geq \frac13  ||\omega-\1_{E^*}||_{L^1}
    \end{align*}
    where $E^*$ is the ball centered at the origin with $|E^*| = \int_{\mathbb{R}^2} \omega(x)dx$.
\end{proposition}
\begin{proposition}\label{prop:ibound}
    Let $I,r>0$. There exists $C= C(I,r) >0$ such that if $\omega:\mathbb{R}^2 \to [0,1]$ satisfies 
    \[
    \int |x|^2 \omega(x)dx \leq I, \int x\omega(x)dx =0, \text{ and } \int \omega(x)dx = \pi r^2,
    \]
    then
    \begin{align*}
        \inf_{a\in\mathbb{R}^2}||\omega-\1_{B(a,r)}||_{L^1} \geq C  ||\omega-\1_{B(0,r)}||_{L^1}^2
    \end{align*}
    Further, the exponent is sharp in this inequality.
\end{proposition}
The proofs are fairly straightforward and are included in the appendix. Our stability theorems for the Rankine vortex patch are an easy consequence of these propositions and the conserved quantities of 2D Euler. We will only present the proof for \cref{thm:mstability}, as the proof for \cref{thm:istability} is essentially identical.

\begin{theorem*}[Restatement of \cref{thm:mstability}]
    There exists a universal constant $C>0$, such that if $\omega_0: \mathbb{R}^2 \to [0,1]$ satisfies $\omega_0 = \omega_0 \circ R_{2\pi/m}$ for some $m \geq 2$ and $\int \omega_0(x)dx = \pi r^2$, then the solution to 2D incompressible Euler having initial data $\omega_0$ satisfies
    \begin{align*}
        ||\omega_t - \1_{B(0,r)}||^2_1 \leq C(E(\1_{B(0,r)}) -E(\omega_0))
    \end{align*}
    for all time.
\end{theorem*}
\begin{proof}
    $\omega_0$ satisfies the conditions of \cref{prop:mbound}, and these conditions are conserved for vorticity solutions to 2D incompressible Euler. It follows that 
    \begin{align*}
        ||\omega_t-\1_{B(0,r)}||_{L^1}^2 \leq 9 \inf_{a\in \mathbb{R}^2} ||\omega_t-\1_{B(a,r)}||^2
    \end{align*}
    From \cref{prop:pseudo-bound}, we know there exists $C' >0$ such that
    \begin{align*}
        \inf_{a\in \mathbb{R}^2} ||\omega_t-\1_{B(a,r)}||^2 \leq C'(E(\1_{B(0,r)})-E(\omega_t))
    \end{align*}
    Altogether, we get
    \begin{align*}
        ||\omega_t-\1_{B(0,r)}||_{L^1}^2 &\leq C(E(\1_{B(0,r)})-E(\omega_t)) \\
        &= C(E(\1_{B(0,r)})-E(\omega_0))
    \end{align*}
    since the pseudo-energy is conserved.
\end{proof}

\section{Perimeter Growth of a Simply Connected Vortex Patch}

Our proof of \cref{thm:pgrowth} follows three steps. First, we use the stability under $m$-fold symmetry given by \cref{thm:mstability} to construct an initial vortex patch with enough angular momentum to ensure large radial support for all time, but that stays close enough to the disk in $L^1$ to power the methods of \cite{Drivas_2024}. They showed that the problem of perimeter growth can be reduced to finding two points on the boundary of $\partial \Omega_t$, sufficiently far from the origin, whose trajectories have each winded around the origin vastly different amounts. Instead of working directly with $\partial \Omega_t$, our second step is to find points in $\Omega^c_t$ that accomplish this difference in winding. The final step is to use the approximation of our velocity field by the purely rotational velocity of the Rankine vortex to show that points in $\Omega^c_t$ are always close in winding number to some point on $\partial \Omega_t$. Putting that all together, we get \cref{thm:pgrowth}.

\begin{theorem*}[Restatement of \cref{thm:pgrowth}]
    There exists a simply connected open set $\Omega_0 \subset \mathbb{R}^2$ (with smooth boundary) such that the vortex patch solution $\omega_t = \1_{\Omega_t}$ to 2D incompressible Euler with initial data $\omega_0 =\1_{\Omega_0}$ satisfies
    \[
        \text{Perimeter}(\Omega_t) \gtrsim t
    \]
    for $t$ sufficiently large.
\end{theorem*}
\subsection{The Setup}
Consider a 3-fold symmetric, simply connected set $\Omega_0$ consisting of a disk $B(0,1^-)$ with three thin arms of length $N>>1$ and combined area of $\gamma<<1$, as depicted in \cref{fig:patch}. Further, we require $\omega_0=\1_{\Omega_0}$ satisfies
\begin{enumerate}
    \item $\int \omega_0 = \pi$
    \item $\int |x|^2 \omega_0=: 1 +\int_{B(0,R)} |x|^2 = 1+\frac{2\pi}{3}R^3 \sim  1+\gamma N^2$
    \item $E(\1_{B(0,1)})-E(\omega_0) \lesssim \gamma \ln(N) =: \delta$
\end{enumerate}

Clearly for any $R>>1$ and $\delta<<1$, there exists $\gamma$ and $N$ so that the last two properties hold. \\

If we take $\omega_0$ as our initial vortex patch, \cref{thm:mstability} gives that for any $\tilde \epsilon > 0$, we can choose $\delta << 1$ so that $\omega_t$ satisfies
\begin{align*}
    ||\omega_t -\1_{B(0,1)}||_{L^1} < \tilde\epsilon
\end{align*}
for all time. Further, conservation of angular momentum gives that for all time,
\begin{align*}
    \Omega_t \setminus B(0,R) \neq \emptyset.
\end{align*}
Following the estimates in Lemma 3.19 of \cite{Drivas_2024}, we can show the following lemma: 
\begin{lemma}\label{lem:velbounds}
    If $u = u_r e_r + ru_\theta e_\theta$ is the velocity induced by $\omega_t$, and $u_* = r\mu(r)e_\theta$ is the velocity induced by $\1_{B(0,1)}$, then:
    \begin{enumerate}
        \item $||u-u_*||_{L^\infty} \lesssim \tilde \epsilon^{1/2}$ \label{eq:uniform}
        \item $||u_\theta||_{L^\infty} \lesssim1 $
        \item $||u_\theta - \mu||_{L^2} \lesssim \ln(R)\tilde\epsilon^{1/4}$
        \item $||\frac1r u_r||_{L^2} \lesssim \ln(R) \tilde\epsilon^{1/4}$
    \end{enumerate}
\end{lemma}
\begin{proof}
    Statement $1$ is standard and follows immediately from
    \begin{align*}
        ||u-u_*||_{L^\infty} &\leq 2(||\omega_t-\1_{B(0,1)}||_{L^1}||\omega_t-\1_{B(0,1)}||_{L^\infty})^{1/2} \\
        &\lesssim \tilde \epsilon^{1/2}
    \end{align*}
    For Statement $2$, we recall from Elgindi \cite{elgindi2016remarksfunctionsboundedlaplacian} that under $3$-fold symmetry, 
    \begin{align}
        |u(x)|\lesssim|x|||\omega_t||_{L^\infty} \label{eq:near}
    \end{align}
    This gives
    \begin{align*}
        ||u_\theta||_{L^\infty} \leq ||u(x)/|x|||_{L^\infty} \lesssim ||\omega_t||_{L^\infty} = 1
    \end{align*}
    For the last two Statements, we recall from \cite{Drivas_2024} that
    \begin{align}
        |u(x)|\lesssim (1+|x|)^{-1} \big(||(1+|y|^2)\omega||_{L^1} + ||\omega||_{L^\infty}\big) \label{eq:far}
    \end{align}
    Note that it suffices to show the stated bounds for $|||x|^{-1}|u-u_*|||_{L^2}$, since
    \begin{align*}
        |\frac1r u_r|, |u_\theta-u| \leq |x|^{-1}|u-u_*|
    \end{align*}
    Picking cutoffs $0<A<B$, we have
    \begin{align*}
        |||x|^{-1}|u-u_*|||^2_{L^2} &= (\int_{|x|<A} + \int_{A<|x|<B} + \int_{|x|>B})|x|^{-2}|u-u_*|^2 
    \end{align*}
    Applying \cref{eq:near}, \cref{eq:uniform}, and \cref{eq:far} to each term respectively, we get
    \begin{align*}
        |||x|^{-1}|u-u_*|||^2_{L^2}&\lesssim A^2 + \tilde\epsilon \int_{A<|x|<B} |x|^{-2}dx + (1+R^3)^2\int_{|x|>B} |x|^{-4}dx\\
        &\lesssim A^2+\tilde \epsilon \ln(B/A) +R^6 B^{-2}
    \end{align*}
    Taking $A = \tilde \epsilon^{1/2}$ and $B = R^6\tilde\epsilon^{-1/2}$, we get
    \begin{align*}
        ||u_\theta-\mu||^2_{L^2} \lesssim \tilde \epsilon \ln(R^6/\tilde\epsilon) \lesssim \ln(R)^2\tilde\epsilon^{1/2},
    \end{align*}
    which completes the proof.
\end{proof}

Let $\Phi$ be the Lagrangian flow map associated to the velocity field $u$. Since $\omega_t$ is 3-fold symmetric, $0$ is a fixed point of $\Phi$. Thus we can view $\Phi$ as a map from $(\mathbb{R}^2\setminus\{0\}) \times \mathbb{R}$ to $\mathbb{R}^2\setminus\{0\}$, and consider the lift $\tilde \Phi = (\Phi_r, \Phi_\theta)$ of $\Phi$ to the universal cover $\mathbb{R} \times \mathbb{R}^+$ of $\mathbb{R}^2 \setminus\{0\}$. For simplicity of notation, when 
\[
    p=(r\cos\theta,r\sin\theta) \in \mathbb{R}^2\setminus\{0\} \text{ with } \theta \in[0,2\pi),
\]
we will take $\tilde \Phi(p,t)$ to mean $\tilde \Phi((\theta,r),t)$. It was shown in \cite{Drivas_2024} that there exists $C=C(u_*)>0$ such that the lift  satisfies
\begin{align*}
    ||\Phi_\theta - \theta - t\mu(\Phi_r)||_{L^2(\mathbb{R}^2\setminus\{0\})} &\leq C t \sqrt{||\frac1r u_r||_{L^2}+||u_\theta - \mu||_{L^2}} 
\end{align*}
Applying \cref{lem:velbounds}, we get
\begin{align*}
    ||\Phi_\theta - \theta - t\mu(\Phi_r)||_{L^2} &\lesssim t\sqrt{\ln(R) \tilde \epsilon^{1/4}} \\
    &=:  \epsilon' \cdot t
\end{align*}

Let $\epsilon :=\max(\tilde \epsilon^{1/2},\epsilon')$ and take $\epsilon << 1$. Note that
\begin{align*}
    ||\Phi_\theta - \theta - t\mu(\Phi_r)||_{L^2} &\lesssim \epsilon \cdot t \\
    ||u-u_*||_{L^\infty} &\lesssim \epsilon
\end{align*}

\subsection{The Twisting}
\begin{proof}[Proof of \cref{thm:pgrowth}]
Our goal is to show for any time $T >> 1$ the existence of points $p, q \in \partial \Omega_0$ satisfying
\begin{enumerate}
    \item $\Phi_\theta(q,T) < c_0T < c_1T < \Phi_\theta(p,T)$
    \item $\Phi_r(p,T), \Phi_r(q,T) > r_0 > 0$
\end{enumerate}
for some absolute constants $c_1 > c_0 > 0, r_0>0$ independent of $T$. The geometric Lemma 3.18 used in \cite{Drivas_2024} would then imply that the segment of $\partial \Omega_T$ connecting $\Phi(p,T)$ to $\Phi(q,T)$ has length at least $2r_0(c_1-c_0)T - 1$, giving linear in time perimeter growth. \\

To that end, at time $T >> 1$, $||\omega_T-\1_{B(0,1)}||_{L^1} < \epsilon$ implies the existence of  $O(1)$ measure sets $A_T, B_T \subset \Omega_T^c$ with $\pi_r(A_T) \subset [1,2]$ and $\pi_r(B_T) \subset [16,17]$. Letting $A_0, B_0$ be the preimage of these sets under $\Phi_T$, we have
\begin{align*}
    |\{p \in A_0: |\Phi_\theta - \theta -T\mu(\Phi_r)| > T/32\}|^{1/2} &\leq \frac{1}{T/32} ||\Phi_\theta - \theta - T\mu(\Phi_r)||_{L^2} \\
    &\lesssim \epsilon
\end{align*}
Recall that 
\[
    \mu(r) = \begin{cases}
        \frac12 \quad & \text{  if } r \leq 1 \\
        \frac1{2r} \quad &\text{ if } r \geq 1
    \end{cases}
\]
Since $\mu(\Phi_r) \in [1/4, 1/2]$ on $A_0$, it follows that, for $\epsilon$ sufficiently small, there exists $p_0 \in A_0$ with 
\begin{align*}
    &\Phi_{\theta} (p_0, T) > T/8 \\
    &1 < \Phi_r(p_0,T) < 20
\end{align*}
Similarly, there exists $q_0 \in B_0$ with 
\begin{align*}
    &\Phi_{\theta} (q_0,T) < T/16 \\
    &1 < \Phi_r(q_0,T)< 20
\end{align*}

Unfortunately, these points lie outside the patch $\Omega_0$ rather than on $\partial \Omega_0$. However, the intuitive expectation is that points outside the patch should not be able to move past the outer arms of the patch that many times, due to the small radial velocity of the field and the small angular velocity at large distances from the origin. This implies that under the flow, some points $p, q \in \partial \Omega_0$ should be both dragged by $p_0$ and held back by $q_0$ respectively. \\

To make the remarks above rigorous, let $z_0 \in \Omega_0^c$, and let $Z_t := \tilde \Phi(z_0,t) \in \mathbb{R} \times \mathbb{R}^+$. At time $t$, we can consider the curves
\begin{align*}
    \partial\tilde\Omega_t^n := \tilde\Phi(\partial \Omega_0,t) + (0, 2\pi n)
\end{align*}
lying in the universal cover, which connect to form a simple curve $\Gamma_t$ in $\mathbb{R} \times \mathbb{R}^+$. Note that the curve $\Gamma_t$ splits the cover into an upper set $\tilde\Omega^c_t$ and a lower set $\tilde\Omega_t$. Define $M^n_t \in \mathbb{R} \times \mathbb{R}^+$ as the first maximizer of $\pi_r(\partial \tilde \Omega^n_t)$ when traversed from left to right. Note that
\begin{align*}
    \Omega_t \setminus B(0,R) \neq \emptyset \implies \pi_r(M^n_t) > R >> 1.
\end{align*}

\begin{figure}[htbp]
  \centering
  \fbox{\includegraphics[width=0.6\linewidth]{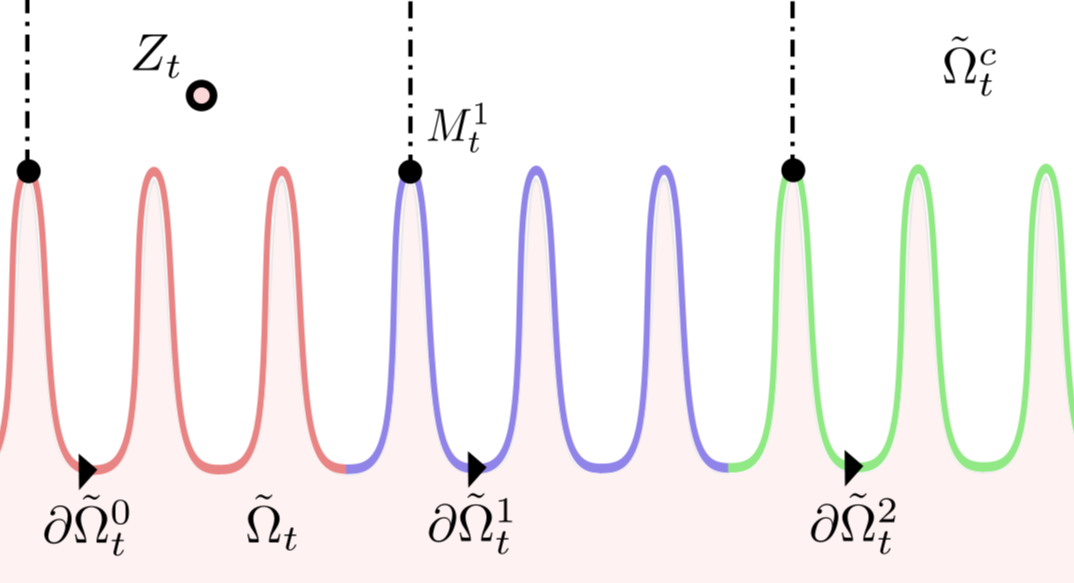}}
  \caption{Illustration of our setup over $\mathbb{R}\times\mathbb{R}^+$. Note that $N(t) = 1$.}
  \label{fig:decomp}
\end{figure}

For the sake of readability, these tildes will be dropped throughout the following discussion. \\

When we draw lines upwards from each $M^n_t$, we separate $\Omega_t^c$ into disjoint buckets. An example of this can be seen in \cref{fig:decomp}. Let $B_t^n$ denote the bucket with $M^n_t$ as its right vertex, and let $N(t)$ denote the label of the bucket containing $Z_t$. (To make $N(t)$ well-defined, a bucket will be considered to contain its right boundary.) \\

The following lemma captures the intuition of the slow rate of passing outer arms of the patch:
\begin{lemma}\label{lem:overtake}
    For any $z_0 \in \Omega_0^c$, the $N(T)$ defined above satisfies
    \begin{align*}
        |N(T)| &\lesssim 1+\max(1/R,\epsilon) T \\
        &\lesssim(\max(1/R,\epsilon)+o(1))T
    \end{align*}
\end{lemma}

Before proving \cref{lem:overtake}, let us first use it to complete the proof of \cref{thm:pgrowth}, i.e., let us show the existence of points $p, q \in \partial \Omega_0$ satisfying
\begin{enumerate}
    \item $\Phi_\theta(q) < c_0t < c_1t < \Phi_\theta(p)$
    \item $\Phi_r(p), \Phi_r(q) > r_0 > 0$
\end{enumerate}
for some absolute constants $c_1 > c_0 > 0$ and $r_0>0$, independent of $t$. For notational convenience, we will use $\nu$ to denote $\max(1/R,\epsilon)$. \\

If we apply the setup of our lemma to the points $p_0 \in A_0$, $q_0 \in B_0$, then $\pi_r(P_T) < R < \pi_r(M^n_T)$ implies that $P_T$ must have a point $\tilde P \in \Gamma_T$, either in period $N_{p_0}(T)-1$ or period $N_{p_0}(T)$, lying directly to its right and at the same height. Similarly, $Q_T$ must have a point $\tilde Q\in \Gamma_T$, either in period $N_{q_0}(T)-1$ or period $N_{q_0}(T)$, lying directly to its left and at the same height. If we let $P$ and $Q$ be the corresponding points on $\partial \Omega_T^0$, then
\begin{align*}
    \pi_\theta (Q) &\leq \pi_\theta(\tilde Q) + 2\pi(|N_{q_0}(T)|+1) \\
    &\leq \pi_\theta(Q_T) + C(\nu+o(1))T \\
    &\leq (1/16+C(\nu+o(1)))T \\ \\
    \pi_\theta(P) &\geq \pi_\theta(\tilde P) - 2\pi(|N_{p_0}(T)+1) \\
    &\geq \pi_\theta(P_T) - C(\nu+o(1))T \\
    &\geq (1/8-C(\nu+o(1)))T \\ \\
    \pi_r(P) &, \pi_r(Q) \geq 1
\end{align*}
Letting $p, q$ be the corresponding points on $\partial \Omega_0$ and noting $\nu << 1 << T$, it's clear that the desired properties are then satisfied, completing the proof of linear in time perimeter growth.
\end{proof}

\begin{proof}[Proof of \cref{lem:overtake}]
    Our goal is to show that over a time scale on the order of $\min(R, \epsilon^{-1})$, the contents of bucket $B^n_t$ stay contained within the buckets corresponding to $n-2 \leq k \leq n+2$. By breaking the interval $[0,T]$ into time steps of this length, we get \cref{lem:overtake}. \\ 
    
    Recall that our velocity field satisfies
    \begin{align*}
        |u_r(x)| &\lesssim \epsilon \\
        |u_\theta(x)| &\lesssim \frac{1}{1+|x|} + \epsilon
    \end{align*}
    It follows that there exists $\Delta \sim \min(R, \epsilon^{-1})$ such that
    \begin{align*}
        \pi_r(\Phi(x,t)) \geq R \implies |\pi_\theta(\Phi(x,t)) - \pi_\theta(\Phi(x,t+\Delta))| \leq 1
    \end{align*}
    Given a bucket $B^n_t$, the following objects are of interest:
    \begin{align*}
        M^n_t&, M^{n-1}_t, \\
        L^n_t &:= \text{the left boundary of }B^n_t \\
        R^n_t &:= \text{the right boundary of }B^n_t \\
        C^n_t &:= \text{the bottom boundary of }B^n_t, \\&\quad\text{ i.e. the segment of }\Gamma_t \text{ connecting } M^{n-1}_t \text{ to }M^n_t
    \end{align*}
    Since we're interested in the time evolution of the bucket, let us define $B^n_t(s)$ to represent the image at time $s$ of the bucket under our flow, and similarly for the other objects. Now at time $s = t + \Delta$, we know that
    \begin{align*}
        C^n_t(s) \subset\partial \Omega^{n-1}_s \cup \partial \Omega^n_s \subset B^{n-1}_s \cup B^n_s \cup B^{n+1}_s
    \end{align*}
    Considering $R^n_t(s)$, we know that it originates from $M^n_t(s)$ and extends to radial infinity. Since it can't pass through $\Gamma_s$, the curve $R^n_t(s)$ exceeds the height of $\Gamma_s$ somewhere within
    \begin{align*}
        \{x : \pi_\theta(M^{n-2}_s) \leq \pi_\theta(x) \leq \pi_\theta(M^{n+1}_s)\}.
    \end{align*}
    Since $\pi_r(R^n_t) \subset(R, \infty)$, our choice of $\Delta$ guarantees
    \begin{align*}
        R^n_t(s) \subset\{x : \pi_\theta(M^{n-2}_s)-2 \leq \pi_\theta(x) \leq \pi_\theta(M^{n+1}_s)+2\}
    \end{align*}
    It follows that $R^n_t(s)$ does not intersect either $L^{n+3}_s$ or $R^{n-3}_s$, and the same argument holds for $L^n_t(s)$. Put all together, we get that $B^n_t(s)$ must be contained within
    \begin{align*}
        B^{n-2}_s \cup B^{n-1}_s \cup B^n_s \cup B^{n+1}_s \cup B^{n+2}_s
    \end{align*}
    Now considering $n=N(t)$, this shows that
    \begin{align*}
        |N(t+\Delta)-N(t)| < 3
    \end{align*}
    Splitting $[0,T]$ into intervals of size $\Delta$, we get that
    \begin{align*}
        |N(T)| &\leq |N(0)| + 3\lceil T/\Delta \rceil \\
        &\lesssim 1 + \max(1/R, \epsilon)T
    \end{align*}
\end{proof}

\section{Appendix}
\begin{proposition*}[Restatement of \cref{prop:pseudo-bound}]
    There exists $C_n>0$ such that for all $\omega : \mathbb{R}^n \to [0,1]$ with $\omega \in L^1(\mathbb{R}^n)$,
    \begin{align*}
        \int\int_{E^*\times E^*}& \ln(|x-y|^{-1})dxdy - \int\int_{\mathbb{R}^n\times\mathbb{R}^n} \omega(x)\ln(|x-y|^{-1})\omega(y)dxdy \\
        &\geq C_n \inf_{a\in\mathbb{R}^n}||\omega-\1_{E^*+a}||_{L^1}^2
    \end{align*}
    where $E^*$ is the ball centered at the origin with $|E^*| = \int_{\mathbb{R}^n} \omega(x)dx$.
\end{proposition*}
\begin{proof}
    We recall the following theorem from \cite{Frank_2021}:
    \begin{theorem}\label{thm:diskbound}
    Let $0 < \delta < 1/2$. Then there is a constant $c_{n,\delta} > 0$ such that for all balls $B \subset \mathbb{R}^n$, centered at the origin, and all $\omega \in L^1(\mathbb{R}^n)$ with $0 \leq \omega \leq 1$ and 
    \begin{align*}
        \delta \leq \frac{|B|^{1/n}}{2||\omega||_{L^1}^{1/n}} \leq 1-\delta,
    \end{align*}
    one has
    \begin{align*}
        \int\int_{E^*\times E^*}& \1_B(x-y)dxdy - \int\int_{\mathbb{R}^n\times\mathbb{R}^n} \omega(x)\1_B(x-y)\omega(y)dxdy \\
        &\geq c_{n,\delta} \inf_{a\in\mathbb{R}^n} ||\omega-\1_{E^*+a}||^2_{L^1}
    \end{align*}
    where $E^*$ is the ball, centered at the origin, of measure $|E^*| = \int_{\mathbb{R}^n} \omega(x)dx$.
    \end{theorem}

    In proving \cref{prop:pseudo-bound}, we can assume
    \begin{align}
        \omega(x)\ln(|x-y|^{-1})\omega(y) \in L^1(\mathbb{R}^n\times \mathbb{R}^n) \label{eq:converge}
    \end{align}
    as otherwise the left-hand side is $\infty$. If we knew the claim for $\omega$ having compact support, then we can conclude for $\omega$ satisfying \cref{eq:converge} by applying the claim to $\omega_n = \omega\1_{B(0,n)}$, taking $n \to \infty$, and noting that all the relevant terms converge. \\

    To that end, if we know $\text{supp }\omega \subset B(0,L)$, then we know that $E^* \subset B(0,L)$. It follows that we can assume $x,y \in B(0,L)$ and write
    \begin{align*}
        \ln(|x-y|^{-1}) = \int_{|x-y|}^{2L} \frac{dR}{R} - \ln(2L)
    \end{align*}
    Using Fubini's Theorem (note that the $\ln(2L)$ terms cancel), we can rewrite the left-hand side as
    \begin{align*}
        \int_{0}^{2L} \frac{dR}{R}\left( \int\int_{E^* \times E^*} \1_{B_R}(x-y)dxdy - \int\int_{\mathbb{R}^n \times \mathbb{R}^n} \omega(x)\1_{B_R}(x-y) \omega(y) dxdy\right)
    \end{align*}
    Let 
    \begin{align*}
        I := \{R > 0: \frac14 \leq \frac{|B_R|^{1/n}}{2||\omega||_{L^1}^{1/n}} \leq \frac34\}
    \end{align*}
    For $R \notin I$, we can lower bound the integrand by $0$ using the Riesz Rearrangement Inequality and the Bathtub Principle. For $R \in I$, we can apply \cref{thm:diskbound} to get an overall lower bound of
    \begin{align*}
        \int\int_{E^*\times E^*}& \ln(|x-y|^{-1})dxdy - \int\int_{\mathbb{R}^2\times\mathbb{R}^2} \omega(x)\ln(|x-y|^{-1})\omega(y)dxdy \\
        &\geq c_{n,\frac14} \inf_{a\in \mathbb{R}^n} ||\omega-\1_{E^*+a}||_{L^1}^2 \int_I \frac{dR}{R},
    \end{align*}
    and it is not hard to see that
    \begin{align*}
        \int_I \frac{dR}{R}
    \end{align*}
    is a constant depending only on the dimension $n$.
\end{proof}

\begin{proposition*}[Restatement of \cref{prop:mbound}]
    Suppose $\omega:\mathbb{R}^2 \to [0,1]$ satisfies $\omega = \omega \circ R_{2\pi/m}$ for some $m \geq 2$, and $\omega \in L^1(\mathbb{R}^2)$. Then
    \begin{align*}
        \inf_{a\in\mathbb{R}^2}||\omega-\1_{E^*+a}||_{L^1} \geq \frac13  ||\omega-\1_{E^*}||_{L^1}
    \end{align*}
    where $E^*$ is the ball centered at the origin with $|E^*| = \int_{\mathbb{R}^2} \omega(x)dx$.
\end{proposition*}
\begin{proof}
    Let
    \begin{align*}
        \int_{\mathbb{R}^2} \omega(x)dx = \pi r^2 \\
        \implies E^* =B(0,r)
    \end{align*}
    and let
    \begin{align*}
        \delta &:= ||\omega-\1_{E^*}||_{L^1} \\
        \epsilon &:= \inf_{a\in\mathbb{R}^2}||\omega-\1_{E^*+a}||_{L^1}
    \end{align*}
    Since $\omega \in L^1$, it is not hard to see that there exists $a \in \mathbb{R}^2$ such that
    \begin{align*}
        \epsilon = ||\omega-\1_{E^*+a}||_{L^1}
    \end{align*}
    Suppose that $m$ is even. Then we know that
    \begin{align*}
        ||\omega-\1_{E^*-a}||_{L^1} = ||\omega-\1_{E^*+a}||_{L^1} = \epsilon
    \end{align*}
    This gives
    \begin{align*}
        \pi r^2 &\geq \int_{B(a,r)\cup B(-a,r)} \omega(x)dx \\
        &= 2(\pi r^2 -\epsilon/2) - \int_{B(a,r)\cap B(-a,r)} \omega(x)dx \\
    \end{align*}
    Noting that $B(a,r)\cap B(-a,r) \subset B(0,r)$, this gives
    \begin{align*}
        \pi r^2 -\epsilon &\leq \int_{B(a,r)\cap B(-a,r)} \omega(x)dx \leq \int_{B(0,r)} \omega(x)dx =\pi r^2 - \delta/2 \\ \\
        \implies& \epsilon \geq \delta/2
    \end{align*}
    showing the claim for $m$ even. For $m = 2n+1$ odd, we can find
    \begin{align*}
        a_1 = a, a_2 = R_{2\pi j/m}(a), a_3=R_{2\pi k/m}(a)
    \end{align*}
    such that
    \begin{align*}
        B(a_1,r)\cap B(a_2,r) \cap B(a_3,r) \subset B(0,r).
    \end{align*}
    Indeed, we can just take $j = n$ and $k=n+1$ and see that
    \begin{align*}
        x &\in B(a_2,r) \cap B(a_3,r) \\
        &\implies x \in B(\frac{a_2+a_3}{2},r) = B(-\cos(\pi/m)a,r) \\
        x&\in B(-\cos(\pi/m)a,r) \cap B(a,r) \\
        &\implies |x| = |\frac{ (x+\cos(\pi/m)a)+\cos(\pi/m)(x-a)}{1+\cos(\pi/m)}| \leq \frac{r+\cos(\pi/m)r }{1+\cos(\pi/m)} = r
    \end{align*}
    Let 
    \begin{align*}
        P_{ij} := B(a_i,r)\cap B(a_j,r),
    \end{align*}
    and note that our computations above give
    \begin{align*}
        \int_{P_{ij}} \omega(x)dx \geq \pi r^2 -\epsilon
    \end{align*}
    This gives
    \begin{align*}
        \pi r^2 &\geq \int_{P_{12}\cup P_{23} \cup P_{13}} \omega(x)dx \\
        &\geq 3(\pi r^2-\epsilon)-2\int_{B(a_1,r)\cap B(a_2,r) \cap B(a_3,r)}\omega(x)dx \\
        &\geq 3(\pi r^2-\epsilon) -2 \int_{B(0,r)}\omega(x)dx \\
        &= 3(\pi r^2 -\epsilon) -2(\pi r^2 -\delta/2) \\ \\
        \implies& \epsilon \geq \delta/3
    \end{align*}
    In either case, we see that
    \begin{align*}
        \inf_{a\in\mathbb{R}^2}||\omega-\1_{E^*+a}||_{L^1} \geq \frac13  ||\omega-\1_{E^*}||_{L^1}
    \end{align*}
\end{proof}
\begin{proposition*}[Restatement of \cref{prop:ibound}]
    Let $I,r>0$. There exists $C= C(I,r) >0$ such that if $\omega:\mathbb{R}^2 \to [0,1]$ satisfies 
    \[
    \int |x|^2 \omega(x)dx \leq I, \int x\omega(x)dx =0, \text{ and } \int \omega(x)dx = \pi r^2,
    \]
    then
    \begin{align*}
        \inf_{a\in\mathbb{R}^2}||\omega-\1_{B(a,r)}||_{L^1} \geq C  ||\omega-\1_{B(0,r)}||_{L^1}^2
    \end{align*}
    Further, the exponent is sharp in this inequality.
\end{proposition*}
\begin{proof}
    Again, let
    \begin{align*}
        \delta &:= ||\omega-\1_{B(0,r)}||_{L^1} \\
        \epsilon &:= \inf_{a\in\mathbb{R}^2}||\omega-\1_{B(a,r)}||_{L^1}
    \end{align*}
    Since $\omega \in L^1$, we find an $a \in \mathbb{R}^2$ such that
    \begin{align*}
        \epsilon = ||\omega-\1_{B(a,r)}||_{L^1}
    \end{align*}
    We have
    \begin{align*}
        \int x\omega(x)dx = 0 \implies \pi r^2 a = \int_{B(a,r)} x (1-\omega(x))dx - \int_{B(a,r)^c}x\omega(x)dx
    \end{align*}
    This gives
    \begin{align*}
        \pi r^2 |a| \leq (|a|+r)\frac{\epsilon}{2} + I^{1/2} (\epsilon/2)^{1/2}
    \end{align*}
    If $\epsilon \geq \pi r^2$, then clearly 
    \begin{align*}
        \epsilon \geq \pi r^2 \geq \delta^2/(4\pi r^2)
    \end{align*}
    Otherwise, we can move the $|a|\epsilon/2$ term to the right-hand side and get
    \begin{align*}
        \frac{\pi r^2}{2}|a| \leq r\epsilon/2 + I^{1/2} (\epsilon/2)^{1/2} \\
        \implies |a| \leq C(I,r) \epsilon^{1/2}
    \end{align*}
    This then gives
    \begin{align*}
        \pi r^2 - \epsilon/2 &=\int_{B(a,r)} \omega(x)dx = \int_{B(a,r) \cap B(0,r)} \omega(x)dx + \int_{B(a,r)\setminus B(0,r)} \omega(x)dx \\
        &\leq \pi r^2 -\delta/2 + |B(a,r)\setminus B(0,r)| \\
        &\leq\pi r^2 -\delta/2 + 2r|a| \\
        &\leq \pi r^2 - \delta/2 +C(I,r) \epsilon^{1/2} \\ \\
        \implies \epsilon &\geq C(I,r)\delta^2
    \end{align*}
    To see that the exponent is sharp, let $\delta << 1$. For the sake of clarity, we will only consider $r=1$. If we let $\omega_0 = \1_{B(\delta e_1,1)}$, then
    \begin{align*}
        ||\omega_0-\1_{B(0,1)}||_{L^1} \sim \delta
    \end{align*}
    To achieve
    \begin{align*}
        \int x\omega = 0,
    \end{align*}
    we can remove a mass of $\delta^2$ from $\omega_0$ and shift it in direction $-e_1$ by a distance on the order of $\delta^{-1}$. Let $\omega$ be the characteristic function of this new set. Then
    \begin{align*}
        \int |x|^2 \omega(x) \lesssim 1,
    \end{align*}
    and by carefully choosing constants in this process, we can ensure that $\omega$ satisfies the conditions of \cref{prop:ibound} with
    \begin{align*}
        \inf_{a\in\mathbb{R}^2}||\omega-\1_{B(a,1)}||_{L^1} \lesssim \delta^2.
    \end{align*}
\end{proof}

\printbibliography

\end{document}